\documentclass[12pt]{article}
\usepackage{amssymb}
\usepackage{amsthm}
\usepackage{amsmath}
\usepackage{mathrsfs}
\usepackage{verbatim}
\begin{document}
\title{On the local structure of optimal measures in the multi-marginal optimal transportation problem\footnote{The author was supported in part by an NSERC postgraduate scholarship.  This work was completed in partial fulfillment of the requirements of a doctoral degree in mathematics at the University of Toronto.}} \author{BRENDAN PASS \footnote{Department of Mathematics, University of Toronto, Toronto, Ontario, Canada, M5S 2E3 bpass@math.utoronto.ca.}}
\maketitle

\begin{abstract}We consider an optimal transportation problem with more than two marginals.  We use a family of semi-Riemannian metrics derived from the mixed, second order partial derivatives of the cost function to provide upper bounds for the dimension of the support of the solution.
\end{abstract}
\section{Introduction}
The optimal transportation problem (with two marginals) asks what is the most efficient way to transform one distribution of mass to another relative to a given cost function.  The problem was originally posed by Monge in 1781 \cite{Mon}.  In 1942, Kantorovich proposed a relaxed version of the problem \cite{Kant}; roughly speaking, he allowed a piece of mass to be split between two or more target points.  Since then, these problems have been studied extensively by many authors and have found applications in such diverse fields as geometry, fluid mechanics, statistics, economics, shape recognition, inequalities, meteorology, etc.

Here we study a multi-marginal generalization of the above; how do we align $m$ distributions of mass with maximal efficiency, again relative to a prescribed cost function.  Precisely, given Borel probability measures $\mu_i$ on smooth manifolds $M_i$ of respective dimensions $n_i$, for $i=1,2...,m$ and a continuous cost function $c: M_1 \times M_2 \times....\times M_m \rightarrow \mathbb{R}$, we would like to minimize 
\begin{equation*}
 C(\gamma)=\int_{M_1 \times M_2 ...\times M_m} c(x_1,x_2,...,x_m)d\gamma 
\end{equation*}
among all measures $\gamma$ on $M_1 \times M_2 ...\times M_m$ which project to the $\mu_i$ under the canonical projections; that is, for any Borel subset $A \subset M_i$,
\begin{equation*}
 \gamma(M_1 \times M_2 \times ....\times M_{i-1}\times A \times M_{i+1}....\times M_m) =\mu_i(A). 
\end{equation*}
When $m=2$, we recover Kantorovich's formulation of the classical optimal transportation problem. 

Under mild conditions, a minimizer $\gamma$ will exist.  Whereas the two marginal problem is relatively well understood, results concerning the structure of these optimal measures have thus far been elusive for $m>2$.  Much of the progress to date has been in the special case where the $M_i$'s are all Euclidian domains of common dimension $n$ and the cost function is given by $c(x_1,x_2,...,x_m)=\sum_{i \neq j}|x_i-x_j|^{2}$, or equivalently $c(x_1,x_2,...,x_m)=-|(\sum_{i}x_i)|^2$.  When $n=3$, partial results for this cost were obtained by Olkin and Rachev \cite{OR}, Knott and Smith \cite{KS} and R\"uschendorf and Uckelmann \cite{RU}, before Gangbo and \'Swi\c ech proved that for a general $m$, under a mild regularity condition on the first marginal, there is unique solution to the Kantorovich problem and it is concentrated on the graph of a function over $x_1$, hence inducing a solution to a Monge type problem \cite{GS}; an alternate proof of Gangbo and \'Swi\c ech's theorem was subsequently found by R\"uschendorf and Uckelmann \cite{RU2}.  This result was then extended by Heinich to cost functions of the form $c(x_1,x_2,...,x_m)=h(\sum_{i}x_i)$ where $h$ is strictly concave \cite{H} and, in the case when the domains $M_i$ are all 1-dimensional, by Carlier \cite{C} to cost functions satisfying a strict 2-monotonicity condition.  More recently, Carlier and Nazaret \cite{CN} studied the related problem of maximizing the determinant (or its absolute value) of the matrix whose columns are the elements $x_1,x_2,...,x_n \in \mathbb{R}^n$; unlike the results in \cite{GS},\cite{H} and \cite{C}, the solution in this problem may not be concentrated on the graph of a function over one of the $x_i$'s and may not be unique.  The proofs of many of these results exploit a duality theorem, proved  in the multi-marginal setting by Kellerer \cite{K}.  Although this theorem holds for general cost functions, it alone says little about the structure of the optimal measure; the proofs of each of the aforementioned results rely heavily on the special forms of the cost. 

  The support of $\gamma$, which we will denote by $spt(\gamma)$, is defined as the smallest closed subset of $M_1 \times M_2\times...\times M_m$ of full mass.  It is natural to ask what this set looks like.  When $m=2$ and the cost function satisfies a twist condition, the solution $\gamma$ is unique and is contained in the graph of a function from $M_1$ to $M_2$, provided the first marginal is suitably regular; this function then solves the original problem posed by Monge \cite{lev}\cite{g}\cite{bren}\cite{GM}\cite{Caf}.  Assuming $M_1$ and $M_2$ are both $C^2$ smooth manifolds of common dimension $n$, the present author, together with McCann and Warren, has shown that under a related non-degeneracy condition on $c$, $spt(\gamma)$  must be contained in an $n$-dimensional Lipschitz submanifold of $M_1 \times M_2$ \cite{MPW}.  For a general $m$, there is an immediate lower bound on the Hausdorff dimension of $spt(\gamma)$; as $spt(\gamma)$ projects to $spt(\mu_i)$ for all $i$, $\text{dim}(spt(\gamma)) \geq \max_{i}(\text{dim}(spt(\mu_i)))$.  In the present manuscript, we establish an upper bound on dim$(spt(\gamma))$. This bound depends on the cost function; however, it will always be greater than the largest of the $n_i$'s.  In the case when the $n_i$'s are equal to some common value $n$, we identify conditions on $c$ that ensure our bound will be $n$ and we show by example that when these conditions are violated, the solution may be supported on a higher dimensional submanifold and may not be unique.  In fact, the costs in these examples satisfy naive multi-marginal extensions of both the twist and non-degeneracy conditions; given the aforementioned results in the two marginal case, we found it surprising that higher dimensional solutions can exist for twisted, non-degenerate costs.  On the other hand, if the support of at least one of the measures $\mu_i$ has Hausdorff dimension $n$, the remarks above imply that $spt(\gamma)$ must be at least $n$ dimensional; therefore, in cases where our upper bound is $n$, the support is exactly $n$-dimensional, in which case we show it is actually $n$-rectifiable.

A striking development in the theory of optimal transportation over the last 15 years has been its interplay with geometry.  Recently, the insight that intrinsic properties of the solution $\gamma$, such as the regularity of Monge solutions, should not depend on the coordinates used to represent the spaces has been very fruitful.  The natural conclusion is that understanding these properties is related to tensors, or coordinate independent quantities.  The relevant tensors encode information about the way that the cost function and the manifolds interact.  For example, Kim and McCann \cite{KM} introduced a pseudo-Riemannian form on the product space, derived from the mixed second order partial derivatives of the cost, whose sectional curvature is related to the regularity of Monge solutions; they also noted that smooth solutions must be timelike for this form. 

Unlike the results of Gangbo and \'Swi\c ech, Heinich and Carlier, our contribution does not rely on a dual formulation of the Kantorovich problem; instead, our method uses an intuitive $c$-monotonicity condition to establish a geometrical framework for the problem.  The question about the dimension of $spt(\gamma)$ should certainly have a coordinate independent answer.  Indeed, inspired partially by Kim and McCann, our condition is related to a family of semi-Riemannian metrics\footnote{For the purposes of this paper, the term semi-Riemannian metric will refer to a symmetric, covariant 2-tensor (which is not necessarily non-degenerate).  The term pseudo-Riemannian metric will be reserved for semi-Riemannian metrics which are also non-degenerate.}; heuristically, $spt(\gamma)$ must be timelike for these metrics and so their signatures control its dimension.  From this perspective, the major difference from the $m=2$ case is that with two marginals, the metric of Kim and McCann always has signature $(n,n)$.  In the multi-marginal case, there is an entire convex family of relevant metrics, generated by $2^{m-1}-1$ extreme points, and their signatures may vary depending on the cost.

Like our work in \cite{MPW} and in contrast to the results of Gangbo and \'Swi\c ech \cite{GS}, Heinich \cite{H}, and Carlier \cite{C}, our results here only concern the local structure of the optimizer $\gamma$ and cannot be easily used to assert uniqueness of $\gamma$ or the existence of a solution to an appropriate Monge type problem.  On the other hand, we do explicitly exhibit fairly innocuous looking cost functions which have high dimensional and non-unique solutions and so it is apparent that these questions cannot be resolved in the affirmative without imposing stronger conditions on $c$.  We address these problems in a separate paper \cite{P}.

The manuscript is organized as follows: in Section 2, we prove our main result.  In Section 3 we apply this result to several example cost functions; many of these are the costs studied by the authors mentioned above and we discuss how they fit into our framework.  In Section 4, we discuss conditions that ensure the relevant metrics have only $n$ timelike directions, which will ensure $spt(\gamma)$ is at most $n$-dimensional.  In Section 5, we discuss some applications of our main result to the two marginal problem and in the final section we take a closer look at the case when the marginals all have one dimensional support.

\textbf{Acknowledgments:}  The author is indebted to his advisor Robert McCann, for originally suggesting this problem to him, for many fruitful discussions and for his continuous support.  He would also like to thank the members of his PhD supervisory committee, Jim Colliander and Bob Jerrard, for providing constructive feedback, and Man-Duen Choi and Chandler Davis for useful discussions concerning the signatures of symmetric matrices.

\section{Dimension of the support}

Before stating our main result, we must introduce some notation.  Suppose that $c \in C^2(M_1 \times M_2 \times...\times M_m)$.  Consider the set $P$ of all partitions of the set $\{1,2,3,...,m\}$ into $2$ disjoint, nonempty subsets; note that $P$ has $2^{m-1}-1$ elements.  For any partition $p \in P$, label the corresponding subsets $p_+$ and $p_-$; thus, $p_+ \cup p_- =\{1,2,3,...,m\}$ and $p_+ \cap p_-$ is empty.  For each $p \in P$, define the symmetric bi-linear form $g_p=\sum_{j \in p_+, k \in p_-} \frac{\partial ^{2} c}{\partial x_{j}^{\alpha_j}\partial x_{k}^{\alpha_k}}(dx_{j}^{\alpha_j}\otimes dx_{k}^{\alpha_k} + dx_{k}^{\alpha_k}\otimes dx_{j}^{\alpha_j})$ on $M_1 \times M_2 ...\times M_m$, where, in accordance with the Einstein summation convention, summation on the $\alpha_k$ and $\alpha_j$ is implicit. 
\newtheorem{cmono}{Definition}[section]
\begin{cmono}
We will say that a subset $S$ of $M_1 \times M_2 \times ...\times M_m$ is $c$-monotone with respect to a partition $p$ if for all  $y=(y_1,y_2,...,y_m)$ and $\tilde{y}=(\tilde{y_1},\tilde{y_2},...,\tilde{y_m})$ in $S$ we have 
\begin{equation*}
c(y) + c(\tilde{y}) \leq c(z) + c(\tilde{z}),
\end{equation*}
where 
\begin{equation*}
z_i =
\begin{cases} y_i & \text{if } i \in p_+,\\
\tilde{y_i} & \text{if } i \in p_-
\end{cases}
\end{equation*}
and
\begin{equation*}
\tilde{z_i} =
\begin{cases} y_i & \text{if } i \in p_-,\\
\tilde{y_i} & \text{if } i \in p_+.
\end{cases}
\end{equation*}
\end{cmono}
The following lemma, which is well known when $m=2$, provides the link between $c$-monotonicity and optimal transportation.
\newtheorem{lemma}[cmono]{Lemma}
\begin{lemma}\label{mono}
Suppose $\gamma$ is an optimizer and  $C(\gamma)< \infty$.  Then the support of $\gamma$ is $c$-monotone with respect to every partition $p \in P$. 
\end{lemma}
\begin{proof}
Define $M_{p_+} =\otimes_{i \in p_+}M_i$ and $M_{p_-} =\otimes_{i \in p_-}M_i$.  Note that we can identify $M_1 \times M_2 \times ...\times M_m$ with $M_{p_+} \times M_{p_-}$ and let $\gamma_{p_+}$ and $\gamma_{p_-}$ be the projections of $\gamma$ onto $M_{p_+}$ and $M_{p_-}$ respectively.  Consider the two marginal problem 
\begin{equation*}
\inf \int_{M_{p_+} \times M_{p_-}}c(x_1,x_2,...,x_m)d\lambda,
\end{equation*}
where the infinum is taken over all measures  $\lambda$ whose projections onto $M_{p_+}$ and $M_{p_-}$ are $\gamma_{p_+}$ and $\gamma_{p_-}$, respectively.  Then $\gamma$ is optimal for this problem and, as $c$ is continuous, the result follows from $c$-monotonicity for two marginal problems; see for example \cite{sk}.
\end{proof}
We will say a vector $v \in T_{(x_1,x_2,...,x_m)}M_1 \times M_2 \times ... \times M_m$ is spacelike (respectively timelike or lightlike) for a semi-Riemannian metric $g$ if $g(v,v) >0$ (respectively $g(v,v)<0$ or $g(v,v)=0$).  We will say a subspace $V \subseteq T_{(x_1,x_2,...,x_m)}M_1 \times M_2 \times ... \times M_m$ is spacelike (respectively timelike or lightlike) for $g$ if every non-zero $v \in V$ is spacelike (respectively timelike or lightlike) for $g$.  We will say $V$ is non-spacelike (respectively non-timelike) for $g$ if no $v \in V$ is spacelike (respectively timelike).  We will say a submanifold of $T_{(x_1,x_2,...,x_m)}M_1 \times M_2 \times ... \times M_m$ is spacelike (respectively timelike, lightlike, non-spacelike or non-timelike) at $(x_1,x_2,...,x_m)$ if its tangent space at $(x_1,x_2,...,x_m)$ is spacelike (respectively timelike, lightlike, non-spacelike or non-timelike).

We are now ready to state our main result:
\newtheorem{metric}[cmono]{Theorem}
\begin{metric} \label{metric}

Let $g$ be a convex combination of the $g_p$'s; that is $g=\sum_{p \in P}t_pg_p$ where $0 \leq t_p \leq 1$ for all $p \in P $ and $\sum_{p \in P}t_p=1$.  Suppose $\gamma$ is an optimizer and $C(\gamma) < \infty$; choose a point $(x_1,x_2,...,x_m) \in M_1 \times M_2 \times ... \times M_m$. Let $N=\sum_{i=1}^mn_i$. Suppose the $(+,-,0)$ signature of $g$ at $(x_1,x_2,...,x_m)$ is $(q_+,q_-,N-q_+-q_-)$ (ie, the corresponding matrix has $q_+$ positive eigenvalues, $q_-$ negative eigenvalues and a zero eigenvalue with multiplicity $N-q_+-q_-$). Then there is a neighbourhood $O$ of $(x_1,x_2,...,x_m)$ such that the intersection of the support of $\gamma$ with $O$ is contained in a Lipschitz submanifold of dimension $N-q_-$.  Wherever the support is smooth, it is non-spacelike for $g$.
\end{metric}  

Before we prove Theorem \ref{metric}, a few remarks are in order.  The roughly says that the dimension of $spt(\gamma)$ is controlled by the signature of \textit{any} convex combinations of the $g_p$'s; as these metrics may have very different signatures, we are free to pick the one with the fewest timelike directions to give us the best upper bound on the dimension of $spt(\gamma)$ for a particular cost.  When $m=2$, there is only one partition in $P$ and consequently there is only one relevant metric, $\frac{\partial ^{2} c}{\partial x_{1}^{\alpha_1}\partial x_{2}^{\alpha_2}}dx_{1}^{\alpha_1}\otimes dx_{2}^{\alpha_2}$ in local coordinates.  The matrix corresponding to this metric is the block matrix studied by Kim and McCann: 
\begin{equation*} \qquad
G=
\begin{bmatrix}
0 & D^{2}_{x_1x_2}c  \\
D^{2}_{x_2x_1}c & 0
\end{bmatrix}.
\end{equation*}
Here $D^{2}_{x_jx_k}c$ is the $n_j$ by $n_k$ matrix whose $(\alpha_j,\alpha_k)$th entry is $\frac{\partial ^{2} c}{\partial x_{j}^{\alpha_j}\partial x_{k}^{\alpha_k}}$.

For $m>2$, in the remainder of this paper we will focus primarily on the special case when $t_p=\frac{1}{2^{m-1}-1}$ for all $p \in P$.   To distinguish it from the metrics obtained by other convex combinations of the $g_p$'s, we will denote the corresponding metric by $\overline{g}$.  Note that the matrix of $\overline{g}$ in local coordinates is the block matrix given by 
\begin{equation*} \qquad
\overline{G}=\frac{2^{m-2}}{2^{m-1}-1}
\begin{bmatrix}
0 & D^{2}_{x_1x_2}c & D^{2}_{x_1x_3}c & ...&D^{2}_{x_1x_m}c \\ 
D^{2}_{x_2x_1}c & 0 & D^{2}_{x_2x_3}c & ...&D^{2}_{x_2x_m}c \\ 
D^{2}_{x_3x_1}c & D^{2}_{x_3x_2}c & 0 & ...&D^{2}_{x_3x_m}c \\
...&...&...&...&...,\\
D^{2}_{x_mx_1}c & D^{2}_{x_mx_2}c & D^{2}_{x_mx_3}c & ...&0
\end{bmatrix}.
\end{equation*}

Let us note, however, that other choices of the $t_p$'s can give new and useful information.  For example, suppose we take $t_p$ to be 1 for a particular $p$ and 0 for all others. As in the proof of Lemma 2.2, we can identify $M_1 \times M_2 ...\times M_m=M_{p_+}\times M_{p_-}$ , where $M_{p_{\pm}}= \otimes_{j \in p_{\pm}}M_j$ and $c(x_1,x_2,...,x_m)=c(x_{p_+},x_{p_-})$ where $x_{p_{\pm}} \in M_{p_{\pm}}$.  In this case, $G$ will take the form:

\begin{equation*} \qquad
G=
\begin{bmatrix}
0 & D^{2}_{x_{p_+}x_{p_-}}c  \\ 
D^{2}_{x_{p_-}x_{p_+}}c &0
\end{bmatrix}.
\end{equation*}

The signature of this $g$ is $(r,r,N-2r)$ where $r$ is the rank of the matrix $D^{2}_{x_{p_+}x_{p_-}}c$.  Letting $n_{p_{\pm}}=\sum_{j \in p_{\pm}}n_j$ be the dimension of $M_{p_{\pm}}$, we will have $r \leq \min(n_{p_+},n_{p_-})$.  If it is possible to choose a partition so that $n_{p_+}=n_{p_-}=\frac{N}{2}$ and  $D^{2}_{x_{p_+}x_{p_-}}c$ has full rank, we can conclude that $spt(\gamma)$ is at most $\frac{N}{2}$ dimensional.  As we will see later, the number of timelike directions for $\overline{g}$ may be very large and so this bound may in fact be better.

Our proof is an adaptation of our argument with McCann and Warren in \cite{MPW}.  When $m=2$, after choosing appropriate coordinates, we rotated the coordinate system and showed that $c$-monotonicity implied that the solution was concentrated on a Lipschitz graph over the diagonal, a trick dating back to Minty \cite{minty}.   When passing to the multi-marginal setting, however, it is not immediately clear how to choose coordinates that make an analogous rotation possible; unlike in the two marginal case, it is not possible in general to choose coordinates around a point $(x_1,x_2,...,x_m)$ such that $D^{2}_{x_ix_j}c(x_1,x_2,...,x_m) = I$ for all $i \neq j$.  The key to resolving this difficulty is the observation that Minty's trick amounts to diagonalizing the pseudo-metric of Kim and McCann and that this approach generalizes to $m\geq3$.
\begin{proof}
Choose a point $x=(x_1,x_2,...,x_m) \in M_1 \times M_2 \times ... \times M_m$ .  Choose local coordinates around $x_i$ on each $M_i$ and set $A_{ij}=D^{2}_{x_ix_j}c(x_1,x_2,...,x_m)$.  For any $\epsilon >0$, there is a neighbourhood $O$ of $(x_1,x_2,...,x_m)$ which is convex in these coordinates such that for all $(y_1,y_2,...,y_m) \in O$ we have $||A_{ij}-D^{2}_{x_ix_j}c(y_1,y_2,...,y_m)|| \leq \epsilon$, for all $i \neq j$.

Let $G$ be the matrix of $g$ at $x$ in our chosen coordinates.  There exists some invertible $N$ by $N$ matrix $U$ such that 
\begin{equation*} \qquad
UGU^{T}=H:=
\begin{bmatrix}
I & 0 & 0\\
0 & -I & 0\\
0 & 0 & 0\\
\end{bmatrix},
\end{equation*}
where the diagonal $I$, $-I$ and $0$ blocks have sizes determined by the signature of $g$.

Define new coordinates in $O$ by $u:=Uy$, where $y=(y_1,y_2,...,y_m)$ and let $u=(u_1,u_2,u_3)$ be the obvious decomposition.  We will show that the optimizer is locally contained in a Lipschitz graph in these coordinates.

Choose $y=(y_1,y_2,...,y_m)$ and $\tilde{y}=(\tilde{y_1},\tilde{y_2},...,\tilde{y_m})$ in the intersection of $spt(\gamma)$ and $O$.  Set $\Delta y=y-\tilde{y}$. Set $z=(z_{1},z_{2},...z_{m})$ where 
\begin{equation*}
z_{i}=
 \begin{cases}
y_i & \text{if }i \in p_+,\\
\tilde{y_i}& \text{if }i \in p_-.
\end{cases}
\end{equation*} 

Similarly, set 
$\tilde{z}=(\tilde{z_{1}},\tilde{z_{2}},...,\tilde{z_{m}})$ where 
\begin{equation*}
\tilde{z_{i}}=
\begin{cases}
y_i & \text{if }i \in p_-, \\
\tilde{y_i}& \text{if }i \in p_+.  
 \end{cases}
\end{equation*}

Lemma \ref{mono} then implies
\begin{equation*}
c(y) + c(\tilde{y}) \leq c(z) + c(\tilde{z})
\end{equation*}
or 
\begin{equation*}
\int_{0}^{1}\int_{0}^{1}\sum_{j \in p_+, i \in p_-}(\Delta y_i)^{T}D^{2}_{x_ix_j}c(y(s,t))\Delta y_k dtds \leq 0,
\end{equation*}

where 
\begin{equation*}
y_i(s,t) =
\begin{cases}
y_i + s(\Delta y_i) & \text{if } i \in p_+,\\
y_i + t(\Delta y_i) & \text{if } i \in p_-.
\end{cases}
\end{equation*}
This implies that 
\begin{eqnarray*}
\sum_{j \in p_+, i \in p_-}(\Delta{y_i})^T A_{ij}\Delta{y_j}\leq \epsilon  \sum_{j \in p_+, i \in p_-}\Delta{y_i}\cdot\Delta{y_j}.
\end{eqnarray*}

Hence,
\begin{eqnarray*}
 \sum _{p \in P}t_p\sum_{j \in p_+, i \in p_-}(\Delta{y_i})^T A_{ij}\Delta{y_j}\leq \epsilon  \sum_{p \in P}t_p\sum_{j \in p_+, i \in p_-}\Delta{y_i} \cdot \Delta{y_j}.
\end{eqnarray*}

But this means 
\begin{equation}\label{timelike}
 (\Delta{y})^TG\Delta{y} \leq \epsilon  \sum_{p \in P} t_p\sum_{j \in p_+, i \in p_-}\Delta{y_i} \cdot \Delta{y_j}.
\end{equation}  
With $\Delta u = U \Delta y$ and $\Delta u=(\Delta u_1, \Delta u_2, \Delta u_3)$ being the obvious decomposition, this becomes:
\begin{align*}
|\Delta u_1|^2-|\Delta u_2|^2=(\Delta{u})^T H \Delta{u}&=(\Delta y)^T G \Delta y \nonumber \\
&\leq\epsilon \sum_{p \in P}t_p\sum_{j \in p_+, i \in p_-}\Delta{y_i} \cdot \Delta{y_j} \nonumber \\
&\leq \epsilon m^{2} ||U^{-1}||^2 \sum_i^3|\Delta u_i|^2, \nonumber
\end{align*}
where the last line follows because for each $i$ and $j$ we have 
\begin{align*}
|\Delta y_i||\Delta y_j| & \leq |\Delta y|^2 \\
&\leq ||U^{-1}||^{2}|\Delta u|^{2} \\
&=||U^{-1}||^{2}\sum_{i=1}^3|\Delta u_i|^2.
\end{align*}
Choosing $\epsilon$ sufficiently small, we have 

\begin{eqnarray*}
 |\Delta u_1|^{2} - |\Delta u_2| ^{2} \leq \frac{1}{2} \sum_{i}^{3} {|\Delta u_i|}^2.
\end{eqnarray*}
Rearranging yields

\begin{eqnarray*}
 \frac{1}{2} |\Delta u_1|^{2} \leq \frac{3}{2} |\Delta u_2|^{2}  + \frac{1}{2} |\Delta u_3|^{2}.
\end{eqnarray*}
Together with Kirzbraun's theorem, the above inequality implies that the support of $\gamma$ is locally contained in a Lipschitz graph of $u_1$ over $u_2$ and $u_3$.  

If $spt(\gamma)$ is differentiable at $x$, the non-spacelike implication follows from taking $y=x$ in (\ref{timelike}), then noting that we can take $\epsilon \rightarrow 0$ as $\tilde{y} \rightarrow x$.

\end{proof}

\section{Examples}

In this section we apply Theorem \ref{metric} to several cost functions.  Throughout this section, the dimensions of the $M_i$ are all equal to some common $n$ and we will restrict our attention to the special semi-Riemannian metric $\overline{g}$ defined in the last section.
\newtheorem{GSH}{Example}[section]  

\begin{GSH} Suppose $M_i=\mathbb{R}^n$ for all $i$ and that $c(x_1,x_2,...,x_m)=h(\sum_{i=1}^m x_i)$, where $D^{2}h < 0$; this is the form of the cost function studied by Gangbo and \'Swi\c ech \cite{GS} and Heinich \cite{H} (actually Heinich made the slightly weaker assumption that $h$ is strictly concave).  Then, up to a positive, multiplicative constant, we have
\begin{equation*} \qquad
\overline{G}=
\begin{bmatrix}
0 & D^{2}h & D^{2}h & ...&D^{2}h \\ 
D^{2}h & 0 & D^{2}h & ...&D^{2}h \\ 
D^{2}h & D^{2}h & 0 & ...&D^{2}h \\ 
...&...&...&...&...,\\
D^{2}h & D^{2}h & D^{2}h & ...&0
\end{bmatrix}.
\end{equation*}

If $v$ is an eigenvector  of $D^{2}h$ with eigenvalue $\lambda$, then 
\begin{equation*} \qquad
[v,v,v,v...v,v]^T
\end{equation*}
is an eigenvector of $\overline{G}$ with eigenvalue $(m-1)\lambda$ and 
\begin{align*} 
[v,-v,0,0,...,0]^T\\
[v,0,-v,0,...,0]^T\\
...........................,\\
...........................,\\
[v,0,0,0,...,-v]^T
\end{align*}
are linearly independent eigenvectors with eigenvalue $-\lambda$.  It follows that the signature of $\overline{g}$ is $((m-1)n,n,0)$ and the solution is contained in an $n$ dimensional submanifold; this is consistent with the results of Gangbo and \'Swi\c ech and Heinich, who show that if the first marginal assigns measure zero to every set of Hausdorff dimension $n-1$, then $spt(\gamma)$ is contained in the graph of a function over $x_1$.

\end{GSH}
\newtheorem{convex}[GSH]{Example}  

\begin{convex} Suppose $c(x_1,x_2,...,x_m)=h(\sum_{i=1}^m x_i)$, but now assume  $D^{2}h > 0$.  Then the signature of $\overline{g}$ is $(n,n(m-1),0)$.  Furthermore, we can show that any measure supported on the $n(m-1)$ dimensional surface 
\begin{equation*}
S=\{(x_1,x_2,...,x_m)|\sum_{i=1}^m x_i = y\},
\end{equation*}
where $y \in \mathbb{R}^{n}$ is any constant, is optimal for its marginals.  Indeed, adding a function of the form $\sum_{i=1}^{m} u_i(x_i)$ to the cost $c$ shifts the functional $C(\gamma)$ by an amount $\sum_{i=1}^{m} \int_{M_i}u_i(x_i)d\mu_i$ for each $\gamma$ but does not change its minimizers.  In particular, minimizing the cost $c$ is equivalent to minimizing 
\begin{equation*}
c'(x_1,x_2,...,x_m):=c(x_1,x_2,...,x_m)-\sum_{i=1}^mx_i\cdot Dh(y)=f(\sum_{i=1}^m x_i),
\end{equation*}
where $f(z):=h(z)-z \cdot Dh(y)$. Then $f$ is a strictly convex function whose gradient vanishes at $z=y$; it follows that $y$ is the unique minimum of $f$.  Hence, $c'(x_1,x_2,...,x_m)\leq f(y)$ with equality only when $\sum_{i=1}^m x_i=y$.  It follows that any measure supported on $S$ is optimal for its marginals.

\end{convex}

\newtheorem{general}[GSH]{Example} 

\begin{general} Let $c(x_1,x_2,...,x_m)=h(\sum_{i=1}^m x_i)$, but now suppose the signature of $D^{2}h$ is $(q,n-q,0)$.  Then the signature of $\overline{g}$ is $(q+(m-1)(n-q),n-q+q(m-1),0)$, and in fact we will find an optimal measures whose support has dimension $(n-q+q(m-1))$.  

At a fixed point $p$, we can add an affine function of $(x_1+x_2+...+x_m)$ so that $Dh(p)=0$ and choose variables so that
\begin{equation*} \qquad
D^{2}h(p)=
\begin{bmatrix}
I & 0 \\
0 & -I \\
\end{bmatrix},
\end{equation*}
where the top left hand corner block is $q$ by $q$ and the bottom left hand corner block is $n-q$ by $n-q$.  Then define the $q$-dimensional variables $y_i = (x_i^1,x_i^2,...,x_i^q)$ and the $n-q$ dimensional variables $z_i = (x_i^{q+1},x_i^{q+2},...,x_i^{n})$, so that $h(\sum_{i=1}^m x_i)=h(\sum_{i=1}^m y_i, \sum_{i=1}^m z_i)$.  Now, near $p$, the implicit function theorem implies that for fixed $z_i, i=1,2,...,m$ there is a unique $K=K(\sum_{i=1}^m z_i))$, such that 
\begin{equation*}
D_yh(K(\sum_{i=1}^m z_i),\sum_{i=1}^m z_i)=0
\end{equation*}
and $K$ is smooth as a function of $\sum_{i=1}^m z_i$.  As $h$ is convex in it's first slot near $p$, 
\begin{equation*}
h(K(\sum_{i=1}^m z_i),\sum_{i=1}^m z_i) \leq h(\sum_{i=1}^m y_i,\sum_{i=1}^m z_i)
\end{equation*}
for all nearby $y_i$.  Now, if we $f(\sum_{i=1}^m z_i)=h(K(\sum_{i=1}^m z_i),\sum_{i=1}^m z_i)$ then $f$ is a concave function of $\sum_{i=1}^m z_i$.  If we consider an optimal transportation problem for the $z_i$ with cost $f$, the solution must be concentrated on a Lipschitz $n-k$ dimensional submanifold.  Choose an $n-q$ dimensional set $S$ which supports an optimizer for this problem; by considering a dual problem as in Gangbo and \'Swi\c ech \cite{GS}, we can find functions $u_i(z_i)$ such that $f(\sum_{i=1}^m z_i) -\sum_{i=1}^m u_i(z_i) \geq 0$ with equality if and only if $(z_1,z_2,...z_m) \in S$.  Therefore,
\begin{equation*}
h(\sum_{i=1}^m y_i, \sum_{i=1}^m z_i))- \sum_{i=1}^m u_i(z_i) \geq  h(K(\sum_{i=1}^m z_i), \sum_{i=1}^m z_i)-\sum_{i=1}^m u_i(z_i) \geq 0
\end{equation*}
and we have equality only when $(z_1,z_2,...z_m) \in S$ and $\sum_{i=1}^m y_i=K(\sum_{i=1}^m z_i)$, which is a $n-q +(m-1)q$ dimensional set.  It follows that this set is the support of an optimizer for appropriate marginals.

\end{general}

\newtheorem{hedonic}[GSH]{Example}  

\begin{hedonic} Chiappori, McCann and Nesheim \cite{CMN} and Carlier and Ekeland \cite{CE} studied a hedonic pricing model involving a multi-marginal optimal transportation problem with cost functions of the form $c(x_1,x_2,...,x_m) = \inf_{y \in Y}\sum_{i=1}^m f_i( x_i,y))$.  Assume:
\begin{enumerate}
\item $Y$ is a $C^2$ smooth $n$-dimensional manifold.
\item For all $i$,  $f_i$ is $C^{2}$ and the matrix $D^{2}_{x_iy}f_i$ of mixed, second order partials derivatives is everwhere non-singular  
\item For each $(x_1,x_2,...,x_m)$ the infinum is attained by a unique $y(x_1,x_2,...,x_m) \in Y$ and
\item $\sum_{i=1}^m D^2_{yy}f_i(x_i,y(x_1,x_2,...,x_m))$ is non-singular.  
\end{enumerate}
Fixing $(x_1,x_2,...,x_m)$, we can choose coordinates so that $D^{2}_{x_iy}f_i(x_i,y(x_1,x_2,...,x_m))=I$ for all $i$.  Now, $\sum_{i=1}^m D_{y}f_i(x_i,y(x_1,x_2,...,x_m))=0$.  Set $M=\sum_{i=1}^m D^2_{yy}f_i(x_i,y(x_1,x_2,...,x_m))$ and note that as $M$ is non-singular by assumption we must have $M > 0$.. The implicit function theorem now implies that $y$ is differentiable with respect to each $x_j$ and:
\begin{equation*}
\sum_{i=1}^m D^2_{yy}f_i(x_i,y(x_1,x_2,...,x_m))D_{x_j}y(x_1,x_2,...,x_m) + D^2_{yx_j}f_j(x_i,y(x_1,x_2,...,x_m))=0.
\end{equation*}
So $D_{x_j}y(x_1,x_2,...,x_m)=-M^{-1}$.  Now, as $c(x_1,x_2,...,x_m) \leq \sum_{i=1}^m f_i( x_i,y))$ with equality when $y=y(x_1,x_2,...,x_m)$ we have
\begin{equation*}
 D_{x_i}c(x_1,x_2,...,x_m)=D_{x_i}f(x_i, y(x_1,x_2,...,x_m)).
\end{equation*}
Differentiating with respect to $x_j$ yields
\begin{equation*}
D^2_{x_ix_j}c(x_1,x_2,...,x_m)=D_{x_iy}f(x_i, y(x_1,x_2,...,x_m))D_{x_j}y(x_1,x_2,...,x_m)=-M^{-1}
\end{equation*}
for all $i\neq j$.  Hence, the signature of $\overline{g}$ is $((m-1)n,n,0)$, by the same argument as in Example 1.
 \end{hedonic}
\newtheorem{Det}[GSH]{Example}  
\begin{Det}The problem studied by Carlier and Nazaret in \cite {CN} is equivalent to the case where $m=n$ and the cost function is the $-1$ times the determinant; ie, for $x_1,x_2,,...,x_n \in \mathbb{R}^n, c(x_1,x_2,...,x_n)$ is $-1$ times the determinant of the $n$ by $n$ matrix whose $ith$ column is the vector $x_i$.  When $n=3$, they exhibit a specific example where the solution has $4$ dimensional support; specifically, it's support is the set 

\begin{equation*} S=\{(x_1,x_2,x_3): |x_1|=|x_2|=|x_3|\text{ and } (x_1,x_2,x_3) \text{ forms a direct orthogonal basis for } \mathbb{R}^3\}.
 \end{equation*}

We show that $\overline{g}$ has signature $(5,4,0)$ on $S$. 
Choose $(x_1,x_2,x_3)$ in the support; after applying a rotation we may assume $x_1=(|x_1|,0,0), x_2=(0,|x_1|,0)$ and $x_3=(0,0,|x_1|)$.  A straightforward calculation then yields:

\begin{equation*} \qquad
\overline{G}=|x_1|
\begin{bmatrix}
0 & 0 & 0 & 0 & -1 & 0 & 0 & 0 & -1 \\
0 & 0 & 0 & 1 & 0 & 0 & 0 & 0 & 0 \\
0 & 0 & 0 & 0 & 0 & 0 & 1 & 0 & 0 \\
0 & 1 & 0 & 0 & 0 & 0 & 0 & 0 & 0 \\
-1 & 0 & 0 & 0 & 0 & 0 & 0 & 0 & -1 \\
0 & 0 & 0 & 0 & 0 & 0 & 0 & 1 & 0 \\
0 & 0 & 1 & 0 & 0 & 0 & 0 & 0 & 0 \\
0 & 0 & 0 & 0 & 0 & 1 & 0 & 0 & 0 \\
-1 & 0 & 0 & 0 & -1 & 0 & 0 & 0 & 0 \\
\end{bmatrix}.
\end{equation*}
There are 5 eigenvectors with eigenvalue 1:
\begin{align*} \qquad
[0 1 0 1 0 0 0 0 0]^T \\
[0 0 1 0 0 0 1 0 0]^T \\
[0 0 0 0 0 1 0 1 0]^T \\
[1 0 0 0 \text{-}1 0 0 0 0]^T \\
[1 0 0 0 0 0 0 0 \text{-}1]^T.
\end{align*}
There are 3 eigenvectors with eigenvalue -1:
\begin{align*} \qquad
[0 1 0 \text{-}1 0 0 0 0 0]^T \\
[0 0 1 0 0 0 \text{-}1 0 0]^T \\
[0 0 0 0 0 1 0 \text{-}1 0]^T .
\end{align*}
Finally, there is a single eigenvalue with eignenvector -2:
\begin{equation*} \qquad
[1 0 0 0 1 0 0 0 1]^T 
\end{equation*}
 \end{Det}

\newtheorem{NU}[GSH]{Example} 

\begin{NU}(Non-uniqueness) Our final example demonstrates that when the dimension of $spt(\gamma)$ is larger than $n$, the solution may not be unique.  Set $m=4$ and $c(x,y,z,w) =h(x+y+z+w)$ for $h$ strictly convex.  Suppose all four marginals $\mu_i$ are Lebesgue measure on the unit cube $I^n$ in $\mathbb{R}^{n}$.   Let $S_1$ be the surface  $y=-w+(1,1,1,...,1)$, $z=-x+(1,1,1,...,1)$ and take $\gamma_1$ be uniform measure on the intersection of $S_1$ with $I^n \times I^n \times I^n \times I^n$.  This projects to $\mu_i$ for $i=1,2,3$ and $4$ and by the argument in Example 2, it must be optimal.  Now, if we take $S_2$ to be the surface $y=-x+(1,1,1,...,1)$, $z=-w+(1,1,1,...,1)$ and $\gamma_2$ be uniform measure on the intersection of $S_2$ with $I^n \times I^n \times I^n \times I^n$, we obtain a second optimal measure.  

It is worth noting that this cost is twisted: the maps $x_i \mapsto D_{x_j}c(x_1,x_2,..x_m)$ are injective for all $i \neq j$ where $x_k$ is held fixed for all $k\neq i$. In the two marginal case, the twist condition and mild regularity on the $\mu_1$ suffices to imply the uniqueness of the solution $\gamma$ \cite{lev}; this example demonstrates that this is no longer true for $m \geq 3$.

\end{NU}

\section{The Signature of $g$}

This section is devoted to developing some results about the signature of the semi-metric $g=\sum_{p \in P}t_pg_p$ at some point $x=(x_1,x_2,...,x_m)$.  Studying the signature at a point reduces to understanding the matrix 

\begin{equation*} \qquad
G=
\begin{bmatrix}
0 & G_{12} & G_{23} & ...&G_{1m} \\ 
G_{21} & 0 & G_{23} & ...&G_{2m} \\  
G_{31} & G_{32} & 0 & ...&G_{3m} \\  
...&...&...&...&...,\\
G_{m1} & G_{m2} & G_{m3} & ...&0
\end{bmatrix}.
\end{equation*}

Here, for $i \neq j, G_{ij}=a_{ij} D^{2}_{x_ix_j}c$ where $a_{ij}=\sum t_p$ and the sum is over all partitions $p \in P$ that separate $i$ and $j$; that is, $i \in p_+$ and $j \in p_-$ or $i \in p_-$ and $j \in p_+$.  One observation about the signature of the matrix $G$ is immediate; as $G$ has zero blocks on the diagonal, it is possible to construct a lightlike subspace of dimension $n_{max} =\text{max}_i\{n_i\}$.  This in turn implies that the number of spacelike directions can be no greater than $N-n_{max}$; otherwise, it would be possible to construct a spacelike subspace of dimension $N-n_{max}+1$, which would have to intersect non trivially with the null subspace. Therefore, the best possible bound on the dimension of $spt(\gamma)$ that Theorem \ref{metric} can provide is $n_{max}$.  This result is not too surprising.  We have already noted that for suitable marginals, the Hausdorff dimension of $spt(\gamma)$ must be at least $n_{max}$; the discussion above verifies that this 
 is consistent with Theorem \ref{metric}.

  Throughout the remainder of this section we will assume the dimensions $n_i$ of the manifolds $M_i$ are all equal to some common $n$.  Then $G$ is an $nm$ by $nm$ matrix; however, we show here that because of its special form,  its signature can be computed from lower dimensional data.  For example, when $m=2$ the signature will always be $(n,n)$ and, as we will see, when $m=3$ it is enough to calculate the signature of an appropriate $n$ by $n$ matrix.  We will refer to the signature of $g$ as $(q_+,q_-,nm-q_+-q_-)$

The first proposition gives an upper and lower bound for the number of timelike directions.   
\newtheorem{upper}{Proposition}[section]

\begin{upper} \label{upper}Suppose $rank(G_{ij}) =r $ for some $i \neq j$.  Then $q_+, q_- \geq r$.  In particular, if $G_{ij}$ is invertible for some $i \neq j$, the support of $\gamma$ is at most $(m-1)n$ dimensional.
\end{upper}

\begin{proof}On the subspace $T_{x_i}M_i \times T_{x_j}M_j$ $G$ restricts to 
\begin{equation*} \qquad
\begin{bmatrix}
0 & G_{ij}\\
G_{ij} & 0\\
\end{bmatrix}.
\end{equation*}
Note that $(v,u)$ is a null vector if and only if $u$ is in the null space of $G_{ij}$ and $v$ is in the nullspace of $G_{ji}$.  As both of these spaces are $n-r$ dimensional, the nullspace of this matrix is $2(n-r)$ dimensional.

As has been noted by Kim and McCann \cite{KM}, the nonzero eigenvalues of this matrix come in pairs of the form $\lambda, -\lambda$, with corresponding eigenvectors $(v,u)$ and $(v,-u)$, respectively, where we take $\lambda \geq 0$ \cite{KM}.  Therefore, there are $\dfrac{1}{2}(2n-2(n-r))=r$ positive eigenvalues and as many negative ones.

We can now construct a $r$ dimensional timelike subspace for $g$.  If $q_+<r$, then we could construct a non-timelike subspace of dimension $mn-q_+>nm-r$ (for example, take the space spanned by all negative and null eigenvalues of $G$).  These two spaces would have to intersect non trivially as their dimensions add to more than $nm$, which is a contradiction.  An analagous argument applies to $q_-$.

\end{proof}

Next, we describe the signature in the $m=3$ case:
\newtheorem{m=3}[upper]{Lemma}

\begin{m=3}\label{m=3}Suppose $m=3$ and that the mixed second order partials are all $G_{12}, G_{13},$ and $G_{23}$ non-singular. Set $A=G_{12}(G_{32})^{-1}G_{31}$; suppose $A+A^{T}$ has signature $(r_+,r_-,n-r_+-r_-)$.  Then $g$ has signature $(q_+,q_-,3n-q_+-q_-)=(n+r_-,n+r_+,n-r_+-r_-)$.  

\end{m=3}
\begin{proof}
By changing variables in $x_2$ and $x_3$, we may assume $G_{12}=G_{13}=I$.  In these coordinates, $G_{32}=A^{-1}$. Consider the subspace 
\begin{equation*}
S=\{(0,p,q): p\in T_{x_2}M_2, q\in T_{x_3}M_3\}.
\end{equation*}
By Proposition~\ref{upper} we can find an orthonormal basis for this subspace consisting of $n$ spacelike and $n$ timelike directions.  To determine the signature of $g$ then, it suffices to consider the restriction of$g$ to the orthogonal complement (relative to $g$) $S^{\perp}$ of $S$; any orthonormal basis of $S^{\perp}$ can be concatenated with a basis for $S$ to form an orthonormal basis for $T_{x_{1}}M_1 \times T_{x_{2}}M_2 \times T_{x_{3}}M_3$.  

A simple calculation yields that $S^{\perp}=\{(v,-A^{T}v,-Av): v \in T_{x_1}M_1\}$ and 

\begin{equation*}
 g((v,-A^{T}v,-Av),(v,-A^{T}v,-Av))=-(A+A^{T})(v,v),
\end{equation*}
which yields the desired result.
\end{proof}

In particular, if $A$ is negative definite, $g$ has signature $(2n,n,0)$ and the support of any minimizer has dimension at most $n$.  

A brief remark about Lemma \ref{m=3} is in order.  We mentioned in section 2 that, while there is only one interesting pseudo metric when $m=2$, there is an entire family of metrics in the $m \geq 3$ setting which may give new information about the behaviour of $spt(\gamma)$.  However, when $m=3$, $D^{2}_{x_ix_j}c$ is non singular for all $i \neq j$, and the coefficients $a_{ij}$ are all non zero, the signature of $g$ is determined entirely by $A=G_{12}(G_{32})^{-1}G_{31}=\dfrac{a_{12}a_{31}}{a_{32}}D^2_{x_1x_2}c(D^2_{x_3x_2}c)^{-1}D^2_{x_3x_1}c$.  Choosing a different $g$ simply changes the $a_{ij}$'s, which does not effect the signature of $A+A^T$.   If one of the $a_{ij}$'s is zero, it is easy to check that the signature of $g$ must be $(n,n,n)$; this yields a bound of $2n$ on the dimension of $spt{\gamma}$ which is no better than the bound obtained when all the $a_{ij}$'s are non-zero.  Thus, when $m=3$ the only information about the dimension of $spt(\gamma)$ which can 
 be provided by Theorem \ref{metric} is encoded in the bi-linear form $D^2_{x_1x_2}c(D^2_{x_3x_2}c)^{-1}D^2_{x_3x_1}c(x)$ on $T_{x_1}M_1 \times T_{x_1} M_1$.

When $m >3$, Lemma \ref{m=3} easily yields the following necessary condition for the signature of $G$ to be $((m-1)n,n,0)$:

\newtheorem{nec}[upper]{Corollary}
\begin{nec}\label{nec}
 Suppose the signature of $g$ is $((m-1)n,n,0)$.  Then 
\begin{equation*}
 D^2_{x_ix_j}c(D^2_{x_kx_j}c)^{-1}D^2_{x_kx_i}c<0
\end{equation*}
for all distinct $i,j$ and $k$.
\end{nec}
\begin{proof}
 Note that the $G_{ij}$'s must be invertible (and hence $D^2_{x_ix_j}c$ must be invertible and $a_{ij} >0$) ; otherwise, the argument in Proposition \ref{upper} implies the existence of a non-spacelike subspace of $T_{x_i}M_i \times T_{x_j}M_j$ whose dimension is greater than $n$.  The signature of $G$ ensures the existence of a $(m-1)n$ dimensional spacelike subspace, however, and so these two spaces would have to intersect non-trivially, a contradiction.

Similarly, if $D^2_{x_ix_j}c(D^2_{x_kx_j}c)^{-1}D^2_{x_kx_i}c$ was not negative definite, we could use Lemma \ref{m=3} to construct a non-timelike subspace of $T_{x_i}M_i \times T_{x_j}M_j \times T_{x_k}M_k$ of dimension greater that $n$; this, in turn, would have to intersect our $(m-1)n$ dimensional timelike subspace, which is again a contradiction.
\end{proof}

The method in the proof above can be extended to give us a method to explicitly calculate the signature of $g$ for larger $m$ when a certain set of matrices are invertible. 

For $l=2,3,...,m$, let $G_l$ be the lower right hand corner $ln$ by $ln$ block of $G$: 

\begin{equation*} \qquad
G_l=
\begin{bmatrix}
0 & G_{m-l+1,m-l+2} & G_{m-l+1,m-l+3} & ...&G_{m-l+1,m} \\ 
G_{m-l+2,m-l+1} & 0 & G_{m-l+2,m-l+3} & ...&G_{m-l+2,m} \\  
G_{m-l+3,m-l+1} & G_{m-l+3,m-l+2} & 0 & ...&G_{m-l+3,m} \\  
...&...&...&...&...,\\
G_{m,m-l+1} & G_{m,m-l+2} & G_{m,m-1+3} & ...&0
\end{bmatrix}.
\end{equation*}
\newtheorem{signature}[upper]{Lemma}

\begin{signature}\label{signature}Suppose $G_l$ has signature $(q,ln-q,0)$   Let $G^l_{ij}$ be the $i,j$th block of the inverse of $G_l$. and consider the $n$ by $n$ matrix $\sum_i G^l_{ij}D^2_{j,l+1}c$.  Suppose this matrix has signature $(r_+,r_-,n-r_+-r_-)$.  Then the signature of $G_{l+1}$ is $(q+r_-,ln-q+r_+,n-r_+-r_-)$.
\end{signature}
For an algorithm to calculate the signature in the general case, start with the lower right hand two by two block, which has signature $(n,n,0)$.  Use Lemma~\ref{signature}, or equivalently Lemma~\ref{m=3} to find the signature of the lower right hand three by three block.  Then use Lemma~\ref{signature} again to determine the signature of the lower right hand four by four block and so on.  After $m-1$ applications of Lemma~\ref{signature} we obtain the signature of $g$.

\section{Applications to the two marginal problem}
Together with McCann and Warren, we proved in \cite{MPW} that any solution to the two marginal problem was supported on an $n$-dimensional Lipschitz submanifold, provided the marginals both live on smooth $n$-dimensional manifolds and the cost is non-degenerate; that is, $D^2_{x_1x_2}c(x_1,x_2)$ seen as a map from $T_{x_1}M_1$ to $T^*_{x_2}M_2$ is injective.  Kim and McCann noted that in this case, the signature of $\overline{g}$ is $(n,n,0)$ \cite{KM}, so Theorem~\ref{metric} immediately implies this result.  In fact, our analysis here is applicable to a larger class of two marginal problems.

Unfortunately, the topology of many important manifolds prohibits the non-degeneracy condition from holding everywhere.  Suppose, for example, that $M_1=M_2=S^1$, the unit circle. Then periodicity in $x_1$ of $\frac{\partial c}{\partial x_2}(x_1,x_2)$ implies
\begin{equation*}
 \int_{S^1}\frac{\partial^2 c}{\partial x_1 \partial x_2}(x_1,x_2)dx_1=0.
\end{equation*}
 It follows that for every $x_2$ there is at least one $x_1$ such that $\frac{\partial^2 c}{\partial x_1 \partial x_2}(x_1,x_2)=0$.  In \cite{MPW}, we noted that under certain conditions the set where non-degeneracy fails is at most $2n-1$ dimensional, which yields an immediate upper bound on the dimension of $spt(\gamma)$.  We now use Theorem~\ref{metric} to derive a better bound.  To this end, suppose that we have two $n$ dimensional manifolds and the non-degeneracy condition fails at some point $(x_1,x_2)$.  If $r$ is the rank of the map $D^{2}_{x_1x_2}c(x_1,x_2)$, then the signature of $\overline{g}$ at $(x_1,x_2)$ is $(r,r,2n-2r)$.   We conclude that locally $spt(\gamma)$ is at most $2n-r$ dimensional.  A global lower bound on $r$ immediately yields an upper bound for the dimension of $spt(\gamma)$

Next we consider a two marginal problem where the dimensions of the spaces fail to coincide; suppose the two manifolds $M_1$ and $M_2$ have dimensions $n_1$ and $n_2$ respectively, where $n_2 \leq n_1$.  Again, let $r$ be the rank of $D^{2}_{x_1x_2}c$; then $\overline{g}$ has signature $(r,r,n_1+n_2-2r)$. If $D^{2}_{x_1x_2}c$ has full rank, ie, if $r=n_2$ then this reduces to $(n_2,n_2,n_1-n_2)$ and the solution may have as many as $n_1$ dimensions (in fact, if the support of the first marginal has Hausdorff dimension $n_1$, then the Hausdorff dimension of $spt(\gamma)$ must be exactly $n_1$).  This result has a nice heuristic explanation.  To solve the problem, one would first solve its dual problem, yielding two potential functions $u_1(x_1)$ and $u_2(x_2)$, and the solutions lies in the set where the first order condition $Du_2(x_2)=D_{x_2}c(x_1,x_2)$ is satisfied.  For a fixed $x_2$, this is a level set of the function $x_1 \mapsto D_{x_2}c(x_1,x_2)$, which is generically $n_1-n_2$ dimensional.  Fixing $x_2$ and moving along this level set corresponds exactly to moving along the null directions of $\overline{g}$.  On the other hand, as $x_2$ varies, $x_1$ must vary in such a way so that the resulting tangent vectors are timelike.  Hence, the solution may contain all the lightlike directions of $\overline{g}$, which correspond to fixing $x_2$ and varying $x_1$, plus $n_2$ timelike directions, which correspond to varying $x_2$ and with it $x_1$.

\section{The 1-dimensional case: coordinate independence and a new proof of Carlier's result}
In \cite{C}, Carlier studied a multi-marginal problem where all the measures were supported on the real line and proved that under a 2-monotonicity condition on the cost, the solution must be one dimensional.  To the best of our knowledge, this is the only result about the multi-marginal problem proved to date that deals with a general class of cost functions.  The purpose of this section is to expose the relationship between 2-monotonicity and the geometric framework developed in this paper.  We will find an invariant form of this condition and provide a new and simpler proof of Carlier's result.

We begin with a definition:

\newtheorem{s2m}{Definition}[section] 

\begin{s2m}We say $c: \mathbb{R}^m \rightarrow \mathbb{R}$ is $i,j$ strictly 2-monotone with sign $\pm 1$ and write $sgn(c)_{ij} = \pm 1$ if for all $x=(x_1,x_2,...,x_m) \in \mathbb{R}^m$ and $s,t>0$ we have

\begin{equation*}
 \pm [c(x)+c(x+te_i+se_j)] < \pm[c(x+te_i) +c(x+se_j)]
\end{equation*}
where $(e_1,e_2,...e_m)$ is the canonical basis for $\mathbb{R}^m$.
\end{s2m}

In this notation, Carlier's 2-monotonicity condition is that $sgn(c)_{ij} =-1$ for all $i \neq j$.  This is not invariant under smooth changes of coordinates, however; the change of coordinates $x_i \mapsto -x_i$ takes a cost with  $sgn(c)_{ij} =-1$ and transforms it to one with $sgn(c)_{ij} =1$.  However, it is easy to check that the following condition is coordinate independent.

\newtheorem{sgn}[s2m]{Definition}

\begin{sgn}We say $c$ is compatible if, for all distinct $i,j,k$ we have 
 \begin{equation}
 \frac{sgn(c)_{ij}sgn(c)_{jk}}{sgn(c)_{ik}} =-1. \nonumber
 \end{equation}
\end{sgn}
It is also easy to check that $c$ is compatible if and only if there exist smooth changes of coordinates $x_i \mapsto y_i=f_i(x_i)$ for $i=1,2,...,m$ which transform $c$ to a 2-monotone cost.  Combined with Carlier's result, this observation implies that compatibility is sufficient to ensure that the support of any optimizer is $1$-dimensional.

If the cost is $C^2$, the condition $\frac{d^2c}{dx_idx_j} <0$ is sufficient to ensure $sgn(c)_{ij} =-1$; likewise, $\frac{d^2c}{dx_idx_j}(\frac{d^2c}{dx_kdx_j})^{-1}  \frac{d^2c}{dx_idx_k} < 0$ ensures that $c$ is compatible.   We can think of the condition on the threefold products $D^2_{x_1x_2}c(D^2_{x_3x_2}c)^{-1}D^2_{x_3x_1}c$ in Lemma \ref{m=3} as a multi-dimensional, coordinate independent version of Carlier's condition.  Corollary \ref{nec} demonstrates that this condition is necessary for $\overline{g}$ to have signature $((m-1)n,n,0)$ and, when $m=3$, Lemma \ref{m=3} shows that it is also sufficient.  For $m >3$, however, it is not sufficient even in one dimension.  As a counterexample, consider the cost function 
 \begin{equation*}
c(x_1,x_2,x_3,x_4)=-x_1x_2 -x_1x_3 - x_1x_4 - x_2x_3 - x_2x_4 - 5x_3x_4.
 \end{equation*}

For this cost, 
\begin{equation*} \qquad
\overline{G}=-
\begin{bmatrix}
0 & 1 & 1 & 1\\
1 & 0 & 1 & 1\\
1 & 1 & 0 & 5\\
1 & 0 & 5 & 0\\
\end{bmatrix},
\end{equation*}
which has signature $(2,2,0)$

Thus, Theorem \ref{metric} implies neither Carlier's result nor the generalization above , at least if we restrict our attention to the special metric $\overline{g}$.  Below, we reconcile this by providing a new proof of Carlier's result, with the slightly stronger assumption $\frac{d^2c}{dx_idx_j} <0$ in place of 2-monotonicity.

The crux of Carlier's argument is the following result:

\newtheorem{Carlier}[s2m]{Theorem}

\begin{Carlier}
Suppose $\frac{d^2c}{dx_idx_j} <0$ for all $i \neq j$ .  Then the projections of the support of the optimizer onto the planes spanned by $x_1$ and $x_j$ are non-decreasing subsets. 
\end{Carlier}
In view of the preceding remarks, this implies that when the cost has negative threefold products $\frac{d^2c}{dx_idx_j}(\frac{d^2c}{dx_kdx_j})^{-1}  \frac{d^2c}{dx_idx_k}$, the support is $1$-dimensional.  

Carlier's proof relies heavily on duality.  He shows that he can reduce the problem to a series of two marginal problems with costs derived from the solution to the dual problem.  He then shows that these cost inherit  monotonicity from $c$ and hence their solutions are concentrated on monotone sets.  We provide a simple proof that uses only the $c$-monotonicity of the support.  In addition, our proof does not require any compactness assumptions on the supports of the measures.  However, after establishing this result, it is not hard to show that, if the first measure is nonatomic, the support is concentrated on the graph of a function over $x_1$.

Morally, our proof applies the non-spacelike conclusion of Theorem \ref{metric} to a well chosen semi-metric; however, because we don't know a priori that the optimizer is smooth we will prove the theorem directly from $c$-monotonicity.

\begin{proof}

Suppose $(x_1,...,x_m)$ and $(y_1,...,y_m)$ belong to the support of the optimizer.  We want to show $(x_1-y_1)(x_i-y_i) \geq 0$ for all $i$.  If not, we may assume without loss of generality that for some $2 \leq k \leq m$ we have $(x_1-y_1)(x_i-y_i) \geq 0$ for all $i<k$ and $(x_1-y_1)(x_i-y_i) < 0$  for $i \geq k$.  Hence, $(x_j-y_j)(x_i-y_i) \leq 0$ for all $j<k$ and $i \geq k$.  By $c$-monotonicity, we have
\begin{equation*}
c(x_1,...,x_m) + c(y_1,...,y_m) \leq c(y_1,...,y_{k-1},x_k,...,x_m) + c(x_1,...,x_{k-1},y_k,...,y_m).
\end{equation*}
Hence,

\begin{eqnarray*}
\sum_{i=1}^{k-1}\sum_{j=k}^{m}(x_i-y_i)(x_j-y_j)\int_{0}^{1}\int_{0}^{1}\frac{d^{2}c}{dx_idx_j}(y_1(t),y_2(t),...,y_{k-1}(t),y_k(s),...,y_m(s))dtds \\
\leq 0
\end{eqnarray*}
where $y_i(t) =y_i + t(x_i-y_i)$ for $i=1,2,..k-1$ and $y_j(s) =y_j + s(x_j-y_j)$ for $j=k,k+1,...,m$.
But, as $\frac{d^{2}}{dx_idx_j}c(y_1(t),y_2(t),...,y_k-1(t),y_k(s),...,y_m(s)) < 0$, and $(x_i-y_i)(x_j-y_j) \leq 0$ for all $i<k$ and $j \geq k$, every term in the sum is nonnegative.  As $(x_1-y_1)(x_j-y_j) <0$ for $j \geq k$, the sum must be positive, a contradiction.

\end{proof}

\end{document}